\documentclass[12pt]{article}

\usepackage{alltt, amsmath, amssymb,amsthm, amsfonts, graphicx, verbatim}

\newtheorem{theorem}{Theorem}
\newtheorem{lemma}[theorem]{Lemma}
\newtheorem{corollary}[theorem]{Corollary}

\newtheorem{remark}[theorem]{Remark}

\title{On the convergence, lock--in probability and sample complexity of stochastic approximation}

\date{}

\begin{document}

\author{Sameer Kamal\footnote{School of Technology and Computer Science,
Tata Institute of Fundamental Research, Homi Bhabha Road,
Mumbai-400005, India. E-mail: sameer.kamal@gmail.com. This work was
supported in part by an Infosys Fellowship.}}

\maketitle

\vspace{.5in}

\noindent \textbf{Abstract:} It is shown that under standard
hypotheses, if stochastic approximation iterates remain tight, they
converge with probability one to what their o.d.e.\ limit suggests.
A simple test for tightness (and therefore a.s. convergence) is
provided. Further, estimates on lock-in probability, i.e., the
probability of convergence to a specific attractor of the o.d.e.\
limit given that the iterates visit its domain of attraction, and
sample complexity, i.e., the number of steps needed to be within a
prescribed neighborhood of the desired limit set with a prescribed
probability, are also provided. The latter improve significantly
upon existing results in that they require a much weaker condition
on the martingale difference noise.

\vspace{.5in}

\noindent \textbf{Key words:} stochastic approximation, tightness of
iterates, almost sure convergence, lock--in probability, sample
complexity

\section{Introduction}

Stochastic approximation was originally introduced in \cite{robbins}
as a scheme for finding zeros of a nonlinear function under noisy
measurements. It has since become one of the main workhorses of
statistical computation, signal processing, adaptive schemes in AI
and economic models, etc. See \cite{benveniste}, \cite{borkar},
\cite{chen}, \cite{duflo}, \cite{kushner} for some recent texts that
give an extensive account. One of the successful approaches for its
convergence analysis has been the `o.d.e.\ approach' of \cite{dere},
\cite{ljung} which treats it as a noisy discretization of an
ordinary differential equation (o.d.e.) with slowly decreasing step
sizes.

The main contributions of this paper are as follows. The first
contribution concerns convergence properties. The aforementioned
convergence analysis is usually of the form: if the iterates remain
a.s.\ bounded, then they converge a.s.\ to a set predicted by the
o.d.e.\ analysis. This a.s.\ boundedness usually has to be
established separately. Here we make the simple observation that
under standard (i.e., commonly assumed) conditions, the tightness of
iterates suffices for a.s.\ convergence to the set predicted by the
o.d.e.\ analysis. A simple test for tightness is also provided.

Our second contribution concerns the lock-in probability, defined as
the probability of convergence to a specific attractor of the
o.d.e.\ if the iterates enter its domain of attraction. Under the
aforementioned standard assumptions, an estimate for this is given
in \cite{borkar}, Chapter 4, p.\ 37, using the Burkholder
inequalities. This has been improved to a much stronger estimate in
\textit{ibid.}, p.\ 41, under the strong hypothesis that suitably
re-scaled martingale difference noise remains bounded. Adapted from
\cite{borkar0}, these results suffice for the application they were
intended for, viz., reinforcement learning algorithms, but are
inadequate for other applications where such a boundedness
hypothesis may be untenable. We recover these results under a much
weaker condition that only requires the re-scaled martingale
differences to have an exponentially decaying conditional tail
probability. Further, we feel that the methodology developed in our
proof is of broader applicability and might prove useful in other
situations.

A third contribution concerns sample complexity. Originating in the
statistical learning theory literature, this notion refers to the
number of samples needed to be within a given precision of the goal
with a given probability. This literature, however,  usually deals
with i.i.d.\ input--output pairs. Here we have a recursive scheme
for which we expect the result to depend upon  the initial position
at iterate $n_0$ (say). Furthermore, the estimate is of an
asymptotic nature, which requires this $n_0$ to be `large enough'
(so that the decreasing step size has decreased sufficiently). Under
the `strong' condition of \cite{borkar0}, this was done in
\cite{borkar0} (see also \cite{borkar}, p.\ 42). We improve on this
by extending the result to the `exponential tail' case mentioned in
the previous paragraph. This, however, is a direct spin-off of the
lock-in probability estimate and follows essentially as in
\cite{borkar0}. As a source of some previous results on exponential
bounds in stochastic approximation we point out \S 6 in the survey
article \cite{walk}, and the literature cited therein.

We prove our `tightness implies convergence' result in section 3
following notational and other preliminaries in section 2. The
simple sufficient condition for tightness is given in section 4.
Section 5, the longest, is devoted to deriving the lock-in
probability estimate from which the sample complexity result of
section 6 follows easily.

\section{Preliminaries}

 Consider the $\mathbb R^d$-valued  stochastic approximation iterates
 \begin{equation}
 x_{n+1} = x_{n} + a(n)[h(x_n)+M_{n+1}], \label{SA}
 \end{equation}
and their `o.d.e.' limit
\begin{equation}
\dot{x}(t) = h(x(t)). \label{ode}
\end{equation}
We make the following assumptions regarding $h(\cdot)$, $a(n)$, and
$M_{n+1}$

\begin{itemize}
\item[\textbf{(A1)}] $h(\cdot): \mathbb R^d \rightarrow \mathbb R^d$ is Lipschitz. Thus
\[
\|h(x) - h(y)\| \leq c\| x - y \| \text{ for some } 0<c<\infty.
\]
\item[\textbf{(A2)}] The step sizes $a(n)$ are positive reals and
satisfy

\begin{itemize}
\item[(i)] $\sum_n a_n=\infty$,
\item[(ii)] $\sum_n a_n^2<\infty$, and
\item[(iii)] $\exists c>0$ such that $a(n) \leq ca(m) \ \forall n \geq m$.
\end{itemize}
\item[\textbf{(A3)}] $(M_n)$ is a martingale difference sequence w.r.t. the filtration
$(\mathcal{F}_n)$ where $\mathcal{F}_n=\sigma(x_0, M_1,\ldots,
M_n)$. Thus, $E[M_{n+1}|\mathcal{F}_n]=0$ a.s. for all $n\geq0$.
Moreover, $M_n$ is square integrable for all $n\geq 0$ with
\begin{equation}
E[\|M_{n+1}\|^2|\mathcal{F}_n]\leq c(1+\|x_n\|^2) \label{martbound}
\end{equation}
a.s. for some $0 < c < \infty$.

\end{itemize}

We next describe the setting for our problem. Let $V:\mathbb
R^d\rightarrow[0,\infty)$ be a differentiable, nonnegative potential
or `Liapunov function' satisfying $\lim_{\|x\|\uparrow\infty}V(x) =
\infty$ and $\dot{V}(x):=\nabla V(x)\cdot h(x) \leq 0 \ \forall x$.
 Define $H:=\{x:\nabla V(x)\cdot h(x) = 0\}$ and assume that this coincides with $\{x:V(x)=0\}$. Note that
$H$ is compact. Under these assumptions, $H$ is an asymptotically
stable, positively invariant set of the limiting o.d.e. (\ref{ode}).
Let $B$ be an arbitrary bounded open set  such that $H\subset B$.
Consider the convergence probability $\mathbb P[x_n\rightarrow
H|x_{n_0}\in B]$ for some  $n_0$. By Theorem 8 of \cite{borkar}, p.\
37, under assumptions (A1)-(A3) the convergence probability
satisfies
\begin{equation}
\mathbb P[x_n\rightarrow H|x_{n_0}\in B] \rightarrow 1 \text{ as }
n_0 \rightarrow \infty. \label{tight-convergence}
\end{equation}
The convergence results of our paper are as follows:
\begin{itemize}
\item If the iterates $\{x_n\}$ are tight and
(\ref{tight-convergence}) holds then the iterates will converge to
$H$ with probability $1$.
\item If the Liapunov function grows exactly quadratically
outside a compact set, then the iterates $\{x_n\}$ are tight.
\end{itemize}

Combining the two, if the Liapunov function grows exactly
quadratically outside a compact set and assumptions (A1)-(A3) hold,
then the iterates will converge to $H$ almost surely.

One is often interested in the `lock--in' probability of a specific
attractor, denoted $H$ again by abuse of notation, of the limiting
o.d.e (\ref{ode}), i.e., the probability of convergence to $H$ given
that the iterates $\{x_n\}$ land up in its domain of attraction
after sufficiently long time. In this spirit,  Theorem 8 of
\cite{borkar}, p.\ 37, shows that $\mathbb P[x_n\rightarrow
H|x_{n_0}\in B] = 1 - O(b(n_0)),$ where $b(n_0) :=
\sum_{m=n_0}^{\infty}a(m)^2$, and $B$ is a bounded open set
contained in the domain of attraction of $H$.
 In this paper we give the following stronger results:

\begin{itemize}
\item Assuming that the scaled martingale difference $\|M_{n+1}\|/(1+\|x_n\|)$ has
 exponentially decaying conditional tail probability, we show that
\[
 \mathbb P\left[x_n\rightarrow H|x_{n_0}\in B\right] = 1 - O\left(e^{-\frac{c}{\sqrt[4]{b(n_0)}}}\right)
\]
  as $n_0 \rightarrow \infty$.
\item As a corollary to the above result we also state a sample complexity result
wherein the `probability of error' is
$O\left(e^{-\frac{c}{\sqrt[4]{b(n_0)}}}\right)$.
\end{itemize}

Similar results are proved in \cite{borkar}, pp.\ 38-41, but under a
much stronger hypothesis, viz., that the scaled martingale
difference sequence above is in fact bounded. This is too
restrictive for many applications.

Finally, before we start our calculations, a remark on notation: in
what follows the letter $c$ may denote a different constant in
different lines. A similar remark applies to the letters $c_1$ and
$c_2$ too.

\section{Convergence for tight iterates}

In this section we relate tightness of the iterates to their almost
sure convergence to $H$. Recall that the iterates $\{x_n\}$ are
tight if given an arbitrary $\epsilon > 0$, there exists a compact
set $K \subset \mathbb R^d$ such that
\[
\mathbb P\left[ x_n \in K \right] > 1 - \epsilon \ \forall n.
\]

\begin{theorem}
\label{result} Assume that the iterates $\{x_n\}$ are tight and  (\ref{tight-convergence}) holds for
any bounded open set $B$ containing $H$. Then
almost surely $x_n \rightarrow H$ as $n \rightarrow \infty$.
\end{theorem}

\begin{proof}
Pick an arbitrary $\epsilon > 0$. Because of tightness there exists
a compact set $K$ such that
\[
\mathbb P\left[ x_n \in K \right] > 1 - \epsilon, \ \forall n.
\]
Now choose a bounded open set $B$ such that $K, H \subset B$. Clearly
\[
\mathbb P\left[ x_n \in B \right] > 1 - \epsilon,  \ \forall n.
\]
Also, by assumption, we have
\[
\mathbb P[x_n\rightarrow H|x_{n_0}\in B] \rightarrow 1 \text{ as }
n_0 \rightarrow \infty.
\]
Combining the two we get
\begin{eqnarray}
\nonumber \mathbb P[x_n\rightarrow H] &\geq& \mathbb P\left[ x_{n_0}
\in B \right]\mathbb P[x_n\rightarrow H|x_{n_0}\in B]\\
\nonumber &>& (1 - \epsilon)\mathbb P[x_n\rightarrow H|x_{n_0}\in
B].
\end{eqnarray}
The left hand side above is independent of $n_0$. Therefore, letting
$n_0\rightarrow \infty$ in the right hand side we get
\[
\mathbb P[x_n\rightarrow H] \geq 1 - \epsilon.
\]
But $\epsilon$ itself was arbitrary. It follows that
\[
\mathbb P[x_n\rightarrow H] = 1.
\]
\end{proof}

\begin{corollary}
Under assumptions (A1)-(A3), if the iterates $\{x_n\}$ are tight
then $x_n \rightarrow H$ a.s.
\end{corollary}

\begin{proof}
This is immediate from Theorem \ref{result} and the fact that
(\ref{tight-convergence}) holds under assumptions (A1)-(A3) by
Theorem 8 of \cite{borkar}, p.\ 37.
\end{proof}

\section{A condition for tightness}

In this section we show that if the Liapunov function grows
`exactly' quadratically outside some compact set then the iterates
are tight. More precisely, we assume that the Liapunov function
$V(\cdot)$ satisfies the following:
\begin{itemize}
\item[\textbf{(A4)}] $V(\cdot)$ is twice differentiable and all second order derivatives are bounded by some
constant. Thus, $|\partial_i\partial_j V(x)| < c$ for all $i, j$ and
$x$.
\item[\textbf{(A5)}] $\|x\|^2 < c(1 + V(x))$ for all $x$ and some $0 < c < \infty$.
\end{itemize}

\begin{theorem}
\label{result2} Under (A4),(A5) the iterates $\{x_n\}$ are tight.
\end{theorem}
\begin{proof}
Without loss of generality, let $E[V(x_0)] < \infty$. Consider
(\ref{SA}), the equation for the iterates. Doing a Taylor expansion
and the using fact that the second order derivatives of $V$ are
bounded, we get
\[
V(x_{n+1}) \leq V(x_n) + a(n)\nabla V(x_n)\cdot[h(x_n)+ M_{n+1}] +
ca(n)^2 \|h(x_n)+ M_{n+1}\|^2.
\]
Since $\nabla V(x_n)\cdot h(x_n)\leq 0$, this yields
\[
V(x_{n+1}) \leq V(x_n) + a(n)\nabla V(x_n)\cdot M_{n+1} + ca(n)^2
\|h(x_n)+ M_{n+1}\|^2.
\]
Lipschitz continuity of $h(\cdot)$ gives us the following bound
\begin{eqnarray}
\nonumber \|h(x_n)+ M_{n+1}\|^2 &=& \|h(x_n)\|^2 + \|M_{n+1}\|^2 + 2
h(x_n)\cdot M_{n+1}\\
\nonumber &<& c\left(1+\|x_n\|^2\right) + \|M_{n+1}\|^2 + 2
h(x_n)\cdot M_{n+1}.
\end{eqnarray}
This leads to
\begin{eqnarray*}
\lefteqn{V(x_{n+1}) < V(x_n) + } \\ &&a(n)\nabla V(x_n)\cdot M_{n+1} + ca(n)^2\left[
\left(1+\|x_n\|^2\right) + \|M_{n+1}\|^2 +  2h(x_n)\cdot
M_{n+1}\right].
\end{eqnarray*}
 Taking conditional
expectation and using (\ref{martbound}) gives
\[
\mathbb E [V(x_{n+1})| \mathcal F_n] < V(x_n) + ca(n)^2 (1 +
\|x_n\|^2).
\]
By (A5), this can be written as
\[
\mathbb E [V(x_{n+1})| \mathcal F_n] < V(x_n) + ca(n)^2 (1+ V(x_n)).
\]
Taking expectations we get
\[
\mathbb E[V(x_{n+1})] < \mathbb E[V(x_n)] + ca(n)^2 (1+ \mathbb
E[V(x_n)]).
\]
This gives
\begin{eqnarray}
\nonumber 1 + \mathbb E[V(x_{n+1})] &<& 1 + \mathbb E[V(x_n)] +
ca(n)^2 (1+ \mathbb E[V(x_n)])\\
\nonumber &=& (1 + ca(n)^2)\left(1 + \mathbb E[V(x_n)]\right)\\
\nonumber &<& \exp(ca(n)^2)\left(1 + \mathbb E[V(x_n)]\right)\\
\nonumber &<& \exp\left(c\sum_{i=0}^\infty a(i)^2\right)\left(1 +
\mathbb E[V(x_0)]\right).
\end{eqnarray}
Since $1 + \mathbb E[V(x_{n+1})]$ is bounded by a constant
independent of $n$, it follows that the iterates are tight.
\end{proof}

\begin{corollary}
Under assumptions (A1)-(A5), we have
\[
\mathbb P[x_n \rightarrow H] = 1.
\]
\end{corollary}
\begin{proof}
This follows from Theorem \ref{result}, Theorem \ref{result2}, and
the fact that (\ref{tight-convergence}) holds under assumptions
(A1)-(A3) by Theorem 8 of \cite{borkar}, p.\ 37.
\end{proof}

\section{Lock-in probability}

In this section we give a lower bound for $\mathbb P[x_n\rightarrow
H|x_{n_0}\in B]$ in terms of $b(n_0)$ when $n_0$ is sufficiently
large. How large $n_0$ needs to be will depend on the choice of $B$,
among other things. Before we proceed further we fix some notation
and recall some known results.

Choose an arbitrary finite $T$ from the interval $(0, \infty)$ and
hold it fixed for the rest of the analysis. Let $t(n)= \sum_{i=0}^{n
- 1}a(i)$. Let $n_0 \geq 0, n_i = \min \{n:t(n)\geq t(n_{i-1})+T\}$.
Define $\bar x(t)$ by: $\bar{x}(t(n)) = x_n$, with linear
interpolation on $[t(n), t(n+1)]$ for all $n$.
 Let $x^{t(n_i)}(\cdot)$ be the solution of the limiting o.d.e. (\ref{ode}) on $[t(n_i), t(n_{i+1}))$ with
 the initial condition $x^{t(n_i)}(t(n_i))=\bar x (t(n_i)) =  x_{n_i}$.
  Let
 \[
 \rho_i := \sup_{t \in [t(n_i), t(n_{i+1}))} \|\bar x(t) - x^{t(n_i)}(t)\|.
 \]

 We recall here a few results from \cite{borkar}. As shown there (\cite{borkar}, pp.\ 32-33), there exists a  $\delta_B > 0$ such that
  if $x_{n_i}\in B$ and $\rho_i < \delta_B$ then $x_{n_{(i+1)}}\in B$,
  too. It is also known (\cite{borkar}, section 2.1, p.\ 16) that if the sequence of iterates
 $\{x_n\}$ remains bounded almost surely on a prescribed set of
 sample points, then it converges almost surely on this
 set to $H$. Combining the two facts gives us the following estimate on the
 probability of convergence, conditioned on $x_{n_0}\in B$ (\cite{borkar}, Lemma 1, p.\ 33)
 \[
 \mathbb{P}\left[\bar x(t)\rightarrow H | x_{n_0}\in B\right] \geq
 \mathbb{P}\left[\rho_i < \delta_B \forall i \geq 0 | x_{n_0}\in B
 \right].
 \]
Let $\mathcal{B}_i$ denote the event that $x_{n_0}\in B$ and
 $\rho_k < \delta_B$ for $k=0,1,\ldots,i$. We get the following lower bound
 for the above probability (\cite{borkar}, Lemma 2, p.\ 33)
 \[
 \mathbb{P}\left[\rho_i < \delta_B \ \forall i \geq 0 | x_{n_0}\in B
 \right]\geq 1-\sum_{i=0}^\infty \mathbb{P}\left[\rho_i \geq \delta_B |\mathcal{B}_{i-1} \right].
 \]
For $n_0$ sufficiently large, this in turn can be bounded as
\[
\mathbb{P}\left[\rho_i \geq \delta_B |\mathcal{B}_{i-1} \right] \leq
\mathbb{P}\left[\max_{0 \leq j < n_{(i+1)}-n_i}
\left.\left\|\sum_{m=0}^{j}a(n_i+m)M_{n_i+m+1}\right\|
> \delta\right|\mathcal{B}_{i-1}\right],
\]
where $\delta=\delta_B/2K_T$, with $K_T$ being a constant that
depends only on $T$ (\cite{borkar}, Lemma 3, p.\ 34).

Thus the probability of convergence, $\mathbb P[x_n\rightarrow
H|x_{n_0}\in B]$, is lower bounded by the following expression
\[
1 - \sum_{i=0}^{\infty}\mathbb{P}\left[\max_{0 \leq j <
n_{(i+1)}-n_i}
\left.\left\|\sum_{m=0}^{j}a(n_i+m)M_{n_i+m+1}\right\|
> \delta\right|\mathcal{B}_{i-1}\right].
\]

In this section we show that $1-\mathbb P[x_n\rightarrow
H|x_{n_0}\in B]$, or the `error probability',  decays exponentially
in $1/\sqrt[4]{b(n_0)}$ provided the scaled martingale difference
terms, $\|M_{i+1}\|/(1+\|x_i\|)$, have  exponentially decaying
conditional tail
 probability. Specifically, we assume that
\begin{equation}
\mathbb{P}\left[\left. \frac{\|M_{i+1}\|}{1 + \|x_i\|} >
v\right|\mathcal F_i\right] \leq C_1\exp(-C_2v), \label{tail}
\end{equation}
for $v$ large enough and for $C_1$ and $C_2$ some positive
constants.

Before we move on to our analysis we introduce a step size
assumption that significantly simplifies our calculations and which
we shall assume for the remainder of this section.

\subsection{A step size assumption}

We  assume that the step sizes $a(i)$ decrease only in `Lipschitz'
fashion. By this we mean that there is a positive constant
$\gamma_T$ depending only on $T$ such that if $a(n_i+m_1)$ and
$a(n_i+m_2)$ are two arbitrary time steps from the same interval
$[t(n_i), t(n_{i+1}))$ then
\begin{equation}
\frac{a(n_i+m_1)}{a(n_i+m_2)} < \gamma_T. \label{stepassume}
\end{equation}

Define $a_{\max} := \sup_n a(n)$. Since $\sum a(n)^2 < \infty$, it
follows that $a_{\max}<\infty$. The next lemma shows
that~(\ref{stepassume}) holds for a large class of step sizes.
\begin{lemma}
\label{reviewer}
 Consider step sizes of the form
\[
a(n)=\frac 1 {n^\alpha(\log n)^\beta},
\]
where either $\alpha\in(1/2,1)$ or $\alpha=1, \beta\leq 0$. For such
step sizes there exists a positive constant $\gamma_T$ depending
only on $T$ such that two arbitrary time steps from the same
interval $[t(n_i), t(n_{i+1}))$ satisfy~(\ref{stepassume}).
\end{lemma}

\begin{proof}
We need to show that for $n_1<n_2$, if $a(n_1) + \cdots + a(n_2) < T
+ a_{\max}$ then for $n_1$ sufficiently large, there exists a
constant $\gamma_T$, depending only on $T$, such that
$a(n_1)/a(n_2)<\gamma_T$. Since $\int_{n_1}^{n_2}a(s)\mathrm ds \leq
a(n_1) + \cdots + a(n_2)$, it suffices to show that there exists a
constant $\gamma_T$ such that, for $n_1$ sufficiently large,
 $\int_{n_1}^{n_2}a(s)\mathrm ds \leq T + a_{\max}$ implies $a(n_1)/a(n_2)<\gamma_T$.
We consider the two cases separately.
\begin{itemize}
\item $\alpha\in(1/2,1)$.

The result follows easily from the following two inequalities which
hold for $n_1$ sufficiently large
\[
\int_{n_1}^{n_2}1/s^\alpha(\log s)^\beta\mathrm ds \geq
\int_{n_1}^{n_2}1/s(\log s)\mathrm ds =
\log{\left(\log{n_2}/\log{n_1}\right)},
\]
and
\[
\int_{n_1}^{n_2}1/s^\alpha(\log s)^\beta\mathrm ds \geq
\int_{n_1}^{n_2}1/s^\nu\mathrm ds =
\frac{n_2^{1-\nu}-n_1^{1-\nu}}{1-\nu} \geq
\frac{(n_2/n_1)^{1-\nu}-1}{1-\nu},
\]
where $\alpha<\nu<1$.
\item $\alpha=1,\beta\leq 0$.

The result follows easily from the following  inequality
\[
\int_{n_1}^{n_2}1/s(\log s)^\beta\mathrm ds \geq
\int_{n_1}^{n_2}1/s\mathrm ds = \log{(n_2/n_1)}.
\]
\end{itemize}

\end{proof}

For $0 \leq m < n_{(i+1)}-n_i-1$, the step size assumption implies
\[
\frac{a(n_i)}{\gamma_T}  \leq a(n_i+m) \leq \gamma_T  a(n_i).
\]
As a result
\[ \frac T {\gamma_T a(n_i)} \left( \frac {a(n_i)}
{\gamma_T} \right)^2 < \sum_{m=0}^{n_{(i+1)}-n_i-1} a(n_i+m)^2  <
\frac{T \gamma_T}{a(n_i)} ({a(n_i)} {\gamma_T} )^2,
\]
whereby
\begin{displaymath}
\label{a(ni)} \sum_{m=0}^{n_{(i+1)}-n_i-1} a(n_i+m)^2 =
\Theta(a(n_i)),
\end{displaymath}
and
\begin{displaymath}
\label{sum a(ni)} b(n_0) = \sum_{i=0}^{\infty}\left(
\sum_{m=0}^{n_{(i+1)}-n_i-1} a(n_i+m)^2 \right) = \Theta\left(
\sum_{i=0}^{\infty}a(n_i) \right).
\end{displaymath}
\begin{remark}
\label{remark} By virtue of the first of the above two equations, we
can use the notationally simpler  term $a(n_i)$ as a proxy for  the
sum $\sum_{m=0}^{n_{(i+1)}-n_i-1} a(n_i+m)^2 $ for obtaining order
estimates. Indeed, in the remainder of this paper we shall
repeatedly do so.
\end{remark}

\subsection{Bounding the error probability}

It will be notationally convenient at this point to introduce
$\zeta_{n_i+j}$ to denote, for an arbitrary $i$, the martingale with
indexing starting at $n_i$ defined as
\[
\zeta_{n_i+j} := \left\{
\begin{array}{cl}
\nonumber 0 & \text{ if } j
= 0,\\
\nonumber  \sum_{m=0}^{j-1}a(n_i+m)M_{n_i+m+1} & \text{ if }  0 < j
\leq n_{(i+1)}-n_i.
\end{array}
\right.
\]
Recall that we seek a bound for $1 - \sum_i \mathbb{P}\left[\max_{0
\leq j \leq n_{(i+1)}-n_i} \|\zeta_{n_i+j}\|
> \delta|\mathcal{B}_{i-1}\right]$. As a first step we bound the
following single term
\[
\mathbb{P}\left[\max_{0 \leq j \leq n_{(i+1)}-n_i} \|\zeta_{n_i+j}\|
> \delta|\mathcal{B}_{i-1}\right].
\]

Our analysis for deriving a bound requires first suitably stopping
the martingale $\zeta_{n_i+j}$, then projecting the stopped
martingale onto a coordinate axis to obtain a  $\mathbb R$-valued
martingale, and finally truncating the difference terms for this
martingale.

Define the stopping time
\[
\tau := \inf\left\{n_i + j:
\left\|\zeta_{n_i+j}\right\|_\infty>\delta/\sqrt d\right\}\wedge
n_{(i+1)}.
\]
Let $\zeta_{n_i+m}^\tau$ denote the stopped martingale
$\zeta_{(n_i+m)\wedge\tau}$. Similarly, let  $M_{n_i+m+1}^\tau$
denote $M_{(n_i+m+1)\wedge\tau}$. We can write
\[
\zeta_{n_i+j}^{\tau} = \sum_{m=0}^{j-1}a(n_i+m)M_{n_i+m+1}^{\tau}.
\]

Let $\mathcal P_z(\cdot)$ denote the projection operator projecting
onto the $z^{\text{th}}$ coordinate. Note that
\begin{eqnarray}
\nonumber && \mathbb{P}\left[\left.\max_{0 \leq j \leq
n_{(i+1)}-n_i} \left\|\zeta_{n_i+j}\right\|
> \delta\right|\mathcal{B}_{i-1}\right]\\
\nonumber &\leq& \mathbb{P}\left[\max_{0 \leq j \leq n_{(i+1)}-n_i}
\left.\left\|\zeta_{n_i+j}\right\|_\infty
> \delta/\sqrt d\right|\mathcal{B}_{i-1}\right]\\
\nonumber &=&
\mathbb{P}\left[\left.\|\zeta_{n_{(i+1)}}^{\tau}\|_\infty
> \delta/\sqrt d\right|\mathcal{B}_{i-1}\right]\\
\label{eqnarray}&\leq& \sum_z \mathbb{P}\left[ \left.\left|\mathcal
P_z\left(\zeta_{n_{(i+1)}}^{\tau}\right)\right|
> \delta/\sqrt d\right|\mathcal{B}_{i-1}\right]
\end{eqnarray}

We'll show that for $n_0$ sufficiently large the following bound
holds.
\[
\mathbb{P}\left[\left. \left| \mathcal P_z
\left(\zeta_{n_{(i+1)}}^{\tau}\right)\right| > \delta/\sqrt d
\right| \mathcal{B}_{i-1}\right] < c_1
\exp\left(-\frac{c\delta^{2/3}}{\sqrt[4]{a(n_i)}}\right).
\]
To derive this bound we'll need a truncated copy of $\mathcal
P_z\left(M_{n_i+m+1}^\tau\right)$. Define $N_{n_i+m+1}$ as follows
\[
N_{n_i+m+1}= \left\{
\begin{array}{cl}
\mathcal P_z\left(M_{n_i+m+1}^\tau\right) & \text{ if } |\mathcal
P_z\left(M_{n_i+m+1}^\tau\right)|\leq v, \\
\mbox{sgn}(\mathcal P_z\left(M_{n_i+m+1}^\tau\right))v & \text{
otherwise },
\end{array} \right.
\]
and define $\eta$ as
\[
\eta := \sum_{m=0}^{n_{(i+1)}-n_i-1} a(n_i+m)N_{n_i+m+1},
\]

Note that
\begin{eqnarray}
\label{error equation} \nonumber && \mathbb{P}\left[\left. \left|
\mathcal P_z \left(\zeta_{n_{(i+1)}}^{\tau}\right)\right| >
\delta/\sqrt d \right| \mathcal{B}_{i-1}\right]\\
\nonumber &\leq& \mathbb P\left[\left.|\eta|
> \delta/\sqrt d \right|
\mathcal{B}_{i-1}\right]\\
 && + \mathbb P\left[\exists m < n_{(i+1)} -n_i : \mathcal
P_z \left(M_{n_i+m+1}^\tau\right) \neq N_{n_i+m+1}|\mathcal
B_{i-1}\right].
\end{eqnarray}

To calculate bounds for the last two terms of (\ref{error equation})
we'll need a bound for the tail probability $\mathbb
P\left[|\mathcal P_z\left(M_{n_i+m+1}^\tau\right)|>u|\mathcal
B_{i-1}\right]$. Let $x_{n_i+m}^{\tau-1}(\cdot)$ denote the
following $F_{n_i+m}$-measurable function
\[
x_{n_i+m}^{\tau-1}(\cdot) = x_{(n_i+m)\wedge(\tau-1)}(\cdot).
\]
In order to get a good bound we first show that for all $i$, and
$m\leq n_{(i+1)}-n_i$, conditional on $\mathcal B_{i-1}$,
$\|x_{n_i+m}^{\tau - 1}\|$ is bounded by a constant.

We shall, therefore, successively get bounds for
\begin{enumerate}
\item
\[
\|x_{n_i+m}^{\tau - 1}(\cdot)\|.
\]
\item
\[
\mathbb P\left[|\mathcal
P_z\left(M_{n_i+m+1}^\tau\right)|>u|\mathcal B_{i-1}\right].
\]
\item
\[
\mathbb P\left[\left. |\eta| > \delta/\sqrt d \right|
\mathcal{B}_{i-1}\right].
\]
\item
\[
\mathbb{P}\left[\left. \left|\mathcal P_z
\left(\zeta_{n_{(i+1)}}^{\tau}\right) \right| > \delta/\sqrt d
\right| \mathcal{B}_{i-1}\right].
\]
\item
\[
\sum_i  \mathbb{P}\left[\max_{0 \leq j \leq n_{(i+1)}-n_i}
\left.\left\|\zeta_{n_i+j}\right\|
> \delta\right|\mathcal{B}_{i-1}\right].
\]
\end{enumerate}

\subsubsection{ Bound for $\|x_{n_i+m}^{\tau - 1}(\cdot)\|$}

Recall that there exists a suitable positive $\delta_B$ such that if
$x_{n_i}\in B$ and $\rho_i < \delta_B$ then $x_{n_{(i+1)}}\in B$,
too. It follows that conditional on $\mathcal B_{i-1}$ we must have
$x_{n_j}\in B$ for $j=0,1,\ldots,i$; in particular, $x_{n_i}\in B$.
Define $K_0 := \sup_{x \in B} \|x\|$. Thus, whatever be the $i$,
conditional on $\mathcal B_{i-1}$ we must have
\[
\|x_{n_i}\| \leq K_0.
\]

We next show that if $\|x_{n_i}\| \leq K_0$ then there exists an $N$
independent of $i$ such that $\|x_{n_i+ m}^{\tau-1}\|<N$ for all
$m\leq n_{(i+1)}-n_i$. As $m$ increases, if $\|x_{n_i+m}^{\tau-1}\|$
is unbounded, then it has to sequentially cross each one of the
values $K_0,K_0+1,\ldots, K_0+n, \ldots $. We will show that for a
fixed, finite $T$ this is not possible. Indeed, we'll show that
there exists a suitable $N$ such that $\|x_{n_i+m}^{\tau-1}\| < N$
for all $m\leq n_{(i+1)}-n_i$ where $N$ does not depend on $i$.  Our
proof will use the fact that the sum $\sum_{i=K_0}^{q} 1/i$ diverges
as $q\rightarrow\infty$.

For $0 \leq m_1 < m_2 \leq n_{(i+1)}-n_{i}-1$ we have
\begin{equation}
\label{pathwise} x_{n_i + m_2} = x_{n_i + m_1} + \sum_{j =
n_i+m_1}^{n_i+m_2-1} a(j)h(x_j) + \sum_{j = n_i+m_1}^{n_i+m_2-1}
a(j)M_{j+1}.
\end{equation}
Let $M_{n_i+m}^{\tau-1}(\cdot)$ denote
$M_{(n_i+m)\wedge(\tau(\omega)-1)}(\omega)$. Note that
$M_{n_i+m}^{\tau-1}(\cdot)$ is not a martingale difference. However,
it is a well defined $\mathcal F_{n_i+m}$-measurable random
variable. Writing (\ref{pathwise}) for the iterates prior to
stopping gives us
\begin{equation}
\label{sample path equation}
 x_{n_i + m_2}^{\tau-1} = x_{n_i +
m_1}^{\tau-1} + \sum_{j = n_i+m_1}^{n_i+m_2-1} a(j)h(x_j)\mathbb
I\{j+1<\tau\} + \sum_{j = n_i+m_1}^{n_i+m_2-1} a(j)M_{j+1}^{\tau-1}.
\end{equation}

For $k\geq 0$ define stopping times $\tau_k$ by
\[
\tau_k(\omega) := \inf\{n_i+m:\|x_{n_i+m}(\omega)\|\geq k\}\wedge
n_{(i+1)}.
\]
By (\ref{sample path equation}),
\[
 x_{\tau_k}^{\tau-1} = x_{\tau_{(k-1)}}^{\tau-1} + \sum_{j = \tau_{(k-1)}}^{\tau_k-1}
a(j)h(x_j)\mathbb I\{j+1<\tau\} + \sum_{j = \tau_{(k-1)}}^{\tau_k-1}
a(j)M_{j+1}^{\tau-1}.
\]

From the definition of $\tau$ it follows that
\[
\sup_{m_1,m_2}\left\| \sum_{m=m_1}^{m_2}a(n_i+m)
M_{n_i+m+1}^{\tau-1}\right\|< 2\delta,
\]
 whenever $m_1$ and $m_2$ are such that $n_i\leq n_i+m_1
\leq n_i+m_2 < n_{(i+1)}$. Further, since $h(\cdot)$ is  Lipschitz,
we have $\|h(x)\| \leq c(1 + \|x\|)$ for some $0<c<\infty$.
Combining the two gives
\begin{eqnarray}
\nonumber &&\|x_{\tau_k}^{\tau-1}\|\\
 \nonumber  &\leq&
 \|x_{\tau_{(k-1)}}^{\tau-1}\| +
\sum_{j=\tau_{(k-1)}}^{\tau_k-1}a(j)c(1+\|x_{j}^{\tau-1}\|) +
\left|\left|\sum_{j=\tau_{(k-1)}}^{\tau_k-1}a(j) M_{j+1}^{\tau-1}\right|\right| \\
\nonumber &\leq&  \|x_{\tau_{(k-1)}}^{\tau-1}\| + \sum_{j =
\tau_{(k-1)}}^{\tau_k-1}a(j)c(1+k) + 2\delta.
\end{eqnarray}

Assume, without loss of generality, that $\delta < 1/6$. If it
isn't, simply replace it by some constant which is less than $1/6$.
Recall that $a_{\max}<\infty$ where $a_{\max} = \sup_n a(n)$. Choose
an $N$ such that
\[
\sum_{k=K_0}^N \frac{1-5\delta}{c(1+k)}> T + a_{\max}.
\]
Let $n_0$ be large enough so that $a(n_0)c(1+N)<\delta$. Assume that
$\|x_{n_i+m}^{\tau-1}\|$ crosses the interval  $[k-1, k]$ from below
$k - 1$ to above $k$ as  $m$ ranges from $0$ to $n_{(i+1)}-n_{i}-1$.
As long as $k\leq N$ it will always be the case that
$\|x_{\tau_{(k-1)}}^{\tau-1}\|$ lies in the range
$[k-1,k-1+3\delta)$. We therefore get
\[
\sum_{j=\tau_{(k-1)}}^{\tau_k-1}a(j)\geq
\frac{\|x_{\tau_k}^{\tau-1}\| -
\|x_{\tau_{(k-1)}}^{\tau-1}\|-2\delta}{c(1+k)} \geq
\frac{1-5\delta}{c(1+k)} ,
\]
as long as $k\leq N$ and $\|x_{n_i+m}^{\tau-1}\|$ crosses the
interval $[k-1, k]$ from below $k - 1$ to above $k$. Since
$\sum_k\sum_{j=\tau_{(k-1)} }^{\tau_k-1}a(j)$ can never exceed
$T+a_{\max}$, and since $\sum_{k=K_0}^N \frac{1-5\delta}{c(1+k)}> T
+ a_{\max}$,  it follows that $N$ is an upper bound for
$\|x_{n_i+m}^{\tau-1}(\cdot)\|$. To summarize:
\begin{lemma}
\label{factoid}
 There exists a constant $N$ such that for all $i$, conditional on $\mathcal B_{i-1}$,
and all $m\leq n_{(i+1)}-n_i$,  the following holds
\[
\|x_{n_i+m}^{\tau - 1}\| < N.
\]
\end{lemma}

\subsubsection{Bound for $\mathbb P\left[|\mathcal
P_z\left(M_{n_i+m+1}^\tau\right)|>u|\mathcal B_{i-1}\right]$}

\begin{lemma}
\label{estimate} There exist constants $K_1$ and $K_2$ such that,
for $u$ sufficiently large, the following holds
\[
\mathbb P\left[|\mathcal P_z(M_{n_i+m+1}^\tau)|>u|\mathcal
B_{i-1}\right] \leq K_1\exp{(-K_2u)}.
\]
\end{lemma}
\begin{proof}
Using first the tail probability bound (\ref{tail}), and then Lemma
\ref{factoid}, we get, for $u$ sufficiently large, the following
bound
\begin{eqnarray}
\nonumber \mathbb P\left[|\mathcal P_z(M_{n_i+m+1}^\tau)|>u|\mathcal
B_{i-1}\right] &\leq& \mathbb
P\left[\left\|M_{n_i+m+1}^\tau\right\|>u|\mathcal B_{i-1}\right] \\
\nonumber &\leq& C_1 \exp\left(-C_2 u/(1+\|x_{n_i+m}^{\tau - 1}\|)\right)\\
\nonumber &\leq& C_1\exp\left(-C_2 u/(1+N)\right).
\end{eqnarray}
\end{proof}

\subsubsection{Bound for $\mathbb P\left[\left. |\eta| > \delta/\sqrt d \right|
\mathcal{B}_{i-1}\right]$}

For $0 \leq m < n_{(i+1)}-n_i$, define $Y_{n_i+m+1}$ as
\[
Y_{n_i+m+1} := N_{n_i+m+1} - \mathbb E [N_{n_i+m+1}|\mathcal
F_{n_i+m}].
\]
Note that $Y_{n_i+1},Y_{n_i+2},\ldots,Y_{n_{(i+1)}}$ is a martingale
difference sequence and, consequently, $\sum_{m=0}^{n_{(i+1)}-n_i-1}
a(n_i+m)Y_{n_i+m+1}$ is a martingale. We can write $\eta$ as
\begin{equation}
\label{decomposition} \eta = \sum_{m=0}^{n_{(i+1)}-n_i-1}a(n_i+m)
Y_{n_i+m+1}+\sum_{m=0}^{n_{(i+1)}-n_i-1}a(n_i+m)\mathbb E
[N_{n_i+m+1}|\mathcal F_{n_i+m}].
\end{equation}

Note that $\mathcal P_z\left(M^\tau_{n_i+m}\right)$ is a martingale
difference for $0 \leq m \leq n_{(i+1)}-n_i$. Using Lemma
\ref{estimate}, this gives us, for $v$ sufficiently large,  the
following bound
\begin{eqnarray} \nonumber E [N_{{n_i}+m+1}|\mathcal F_{n_i+m}] &=& \mathbb E
[N_{{n_i}+m+1}-\mathcal
P_z\left(M^\tau_{n_i+m+1}\right)|\mathcal F_{n_i+m}]\\
\nonumber &\leq& \int_v^\infty
K_1\exp{(-K_2u)}\mathrm du\\
\nonumber &=& c_1\exp{(-cv)}.
\end{eqnarray}
A similar calculation  shows
\[
\mathbb E [N_{{n_i}+m+1}|\mathcal F_{n_i+m}] \geq - c_1\exp{(-cv)},
\]
and consequently, for $0 \leq m \leq n_{(i+1)}-n_i$,
\[
\left|\mathbb E [N_{{n_i}+m+1}|\mathcal F_{n_i+m}]\right| \leq
c_1\exp{(-cv)}.
\]
Combining everything gives
\[
\left|\sum_{m=0}^{n_{(i+1)}-n_i-1}a(n_i+m)\mathbb E
[N_{n_i+m+1}|\mathcal F_{n_i+m}]\right| \leq (T+1)c_1\exp{(-cv)}.
\]

Note that the last expression can be made as small as desired by
choosing $v$ sufficiently large. Choose $v =
\sqrt[3]{\delta^2/a(n_i)}$. The reason for this specific choice will
become clear later. It follows that for $n_0$ large enough, $v$ will
indeed be  as large as required. Assume  $n_0$ to be sufficiently
large that
\begin{equation}
\label{mean} \left|\sum_{m=0}^{n_{(i+1)}-n_i-1}a(n_i+m)\mathbb E
[N_{n_i+m+1}|\mathcal F_{n_i+m}]\right|< \frac{\delta}{2\sqrt d}.
\end{equation}

Using  (\ref{decomposition}) and (\ref{mean}), we get, for $n_0$
sufficiently large, the following
\[
\mathbb P\left[ |\eta|
> \delta/\sqrt d |\mathcal{B}_{i-1}\right]  \leq \mathbb P\left[ \left.\left|\sum_{m=0}^{n_{(i+1)}-n_i-1}a(n_i+m)
Y_{n_i+m+1}\right| > \frac{\delta}{2\sqrt d} \right|
\mathcal{B}_{i-1}\right]
\]

 We recall the Azuma-Hoeffding inequality for martingales that have bounded
differences. Suppose $\{X_k : k = 0, 1, 2,\ldots\}$ is a martingale
and the differences satisy $| X_k - X_{k - 1} | < c_k$ a.s. Then for
all positive integers $n$ and all positive reals $t$, $ \mathbb
P(X_n - X_0 \geq t) \leq \exp\left ({-t^2 \over 2
\sum_{k=1}^{n}c_k^2} \right)$. We'll use it in the two sided form
\[
\mathbb P[|X_n - X_0| \geq t] \leq 2\exp\left
(\frac{-t^2}{2\sum_{k=1}^{n}c_k^2}\right).
\]

Note that $|Y_{n_i+m+1}|\leq 2v$ and $a(n_i+m)\leq\gamma_T a(n_i)$.
Also, $n_{(i+1)}-n_i\leq\gamma_TT/a(n_i)$. This gives, for $n_0$
large enough to satisfy (\ref{mean}), the following bound
\begin{equation}
\mathbb P\left[ \left.\left|\sum_{m=0}^{n_{(i+1)}-n_i-1}a(n_i+m)
Y_{n_i+m+1}\right| > \frac{\delta}{2\sqrt d} \right|
\mathcal{B}_{i-1}\right]\leq 2\exp{\left(-\frac{c\delta^2}{a(n_i)
v^2}\right)}.
\end{equation}

To summarize:
\begin{lemma}
\label{error1}
 For $n_0$ sufficiently large, there exists a constant
$c$ such that
\[
 \mathbb P\left[ |\eta|
> \delta/\sqrt d |\mathcal{B}_{i-1}\right] \leq 2\exp{\left(-\frac{c\delta^2}{a(n_i)
v^2}\right)},
\]
where $v = v(a(n_i))$ is such that $v(a(n_i))\uparrow\infty$ as
$a(n_i)\downarrow 0$.
\end{lemma}

\subsubsection{Bound for $\mathbb{P}\left[\left. \left| \mathcal P_z \left(\zeta_{n_{(i+1)}}^{\tau}\right)\right| >
\delta/\sqrt d \right| \mathcal{B}_{i-1}\right]$}

We have
\begin{eqnarray}
\nonumber &&\mathbb P\left[\exists m < n_{(i+1)} -n_i : \mathcal P_z
\left(M_{n_i+m+1}^\tau\right) \neq N_{n_i+m+1}|\mathcal
B_{i-1}\right]\\
\nonumber &\leq& \frac{\gamma_T T}{a(n_i)}\times K_1\exp{(-K_2v)}\\
\label{error2} &=& \frac{c_1}{a(n_i)}\exp{(-c_2v)}.
\end{eqnarray}
Plugging  (\ref{error2}) in (\ref{error equation}), and applying
Lemma \ref{error1} we get
\begin{eqnarray}
\nonumber && \mathbb{P}\left[\left. \left| \mathcal P_z
\left(\zeta_{n_{(i+1)}}^{\tau}\right)\right| >
\delta/\sqrt d \right| \mathcal{B}_{i-1}\right] \\
\nonumber &\leq& 2\exp{\left(-c\delta^2/a(n_i)
v^2\right)}+\frac{c_1}{a(n_i)}\exp{(-c_2v)}.
\end{eqnarray}

Since the left hand side is independent of $v$ we can choose a
 value for $v$ which keeps the right hand side sufficiently
low. Specifically, we choose
\[
v = \sqrt[3]{\delta^2/a(n_i)}.
\]
This gives us the
following bound
\begin{eqnarray}
\nonumber && \mathbb{P}\left[\left. \left| \mathcal P_z
\left(\zeta_{n_{(i+1)}}^{\tau}\right)\right| >
\delta/\sqrt d \right| \mathcal{B}_{i-1}\right]\\
\nonumber &<& \frac{c_1}{a(n_i)}
\exp\left(-\frac{c\delta^{2/3}}{\sqrt[3]{a(n_i)}}\right) \\
\nonumber &<& c_1
\exp\left(-\frac{c\delta^{2/3}}{\sqrt[4]{a(n_i)}}\right),
\end{eqnarray}
for $n_0$ large enough.

\subsubsection{Bound for $\sum_i  \mathbb{P}\left[\max_{0 \leq j \leq n_{(i+1)}-n_i}
\left.\left\|\zeta_{n_i+j}\right\|
> \delta\right|\mathcal{B}_{i-1}\right]$}

Plugging the last bound  in~(\ref{eqnarray}) we get
\[
\mathbb{P}\left[\left.\max_{0 \leq j \leq n_{(i+1)}-n_i}
\left\|\zeta_{n_i+j}\right\|
> \delta\right|\mathcal{B}_{i-1}\right] \leq c_1
\exp\left(-\frac{c\delta^{2/3}}{\sqrt[4]{a(n_i)}}\right),
\]
where we have absorbed the multiplicative factor of $d$ in the
constant $c_1$.

We note that $c_1 \exp\left(-\frac{
c\delta^{2/3}}{\sqrt[4]{y}}\right)$ is a convex function for $y\in
(0,c_2)$, where $c_2$ is a sufficiently small positive constant.
Furthermore,
\[
c_1 \exp\left(-\frac{ c\delta^{2/3}}{\sqrt[4]{y}}\right)\rightarrow
0 \text{ as } y\downarrow 0.
\]
For such functions we have the following fact:
\begin{lemma}
\label{convex} Let $g(\cdot)$ be a function such that $g(0)=0$ and
$g(\cdot)$ is convex in the region $(0,c)$ for some $c>0$. For $a,b
\geq 0$ and $a+b<c$, the following holds
\[
g(a) + g(b) \leq g(a+b).
\]
\end{lemma}
\begin{proof}
We have
 \begin{eqnarray}
\nonumber g(a) &=& g\left(\frac b {a+b}0 + \frac a {a+b}(a+b)\right) \\
\nonumber &\leq& \frac b {a+b}g(0) + \frac a {a+b}g(a+b) \\
\nonumber &=& \frac a {a+b}g(a+b).
 \end{eqnarray}
Similarly
\[
g(b) \leq \frac b {a+b}g(a+b).
\]
Adding the two we get
\[
g(a)+g(b) \leq g(a+b).
\]
\end{proof}

Finally, by Lemma~\ref{convex} and Remark~\ref{remark}, we get
\begin{eqnarray}
\nonumber \sum_i  \mathbb{P}\left[\max_{0 \leq j \leq n_{(i+1)}-n_i}
\left.\left\|\zeta_{n_i+j}\right\|
> \delta\right|\mathcal{B}_{i-1}\right] &<& \sum_i c_1
\exp\left(-\frac{c\delta^{2/3}}{\sqrt[4]{a(n_i)}}\right)\\
\nonumber &\leq& c_1
\exp\left(-\frac{c\delta^{2/3}}{\sqrt[4]{\sum_ia(n_i)}}\right)\\
\nonumber &\leq& c_1
\exp\left(-\frac{c\delta^{2/3}}{\sqrt[4]{b(n_0)}}\right),
\end{eqnarray}
provided $n_0$ is sufficiently large. Note, in particular, that
$n_0$ should be large enough to ensure that $\sum_ia(n_i)$ lies in
the region of convexity of
$c_1\exp\left(-\frac{c\delta^{2/3}}{\sqrt[4]{y}}\right)$.

To summarize the calculations of this section, we have proved the
following result:

\begin{theorem}
Under assumptions (A1)-(A3),  the assumption that for large $u$,
 the tail probability bound (\ref{tail}) holds, and the step size assumption
(\ref{stepassume}), we have the following bound provided $n_0$ is
sufficiently large
\[
\mathbb P[x_n\rightarrow H|x_{n_0}\in B] \geq 1 - c_1
\exp\left(-\frac{c\delta^{2/3}}{\sqrt[4]{b(n_0)}}\right),
\]
where $\delta=\delta_B/2K_T$.
\end{theorem}

\section{Application: a sample complexity result}

As an application of our result we give here a sample complexity
estimate, which roughly says that conditional on $x_{n_0}\in B$ for
some fixed, sufficiently large $n_0$, with a high probability the
interpolated trajectory $\bar x(t)$ will be sufficiently close to
$H$ after any lapse of time greater than some fixed $\gamma$. We now
state the result more formally. We briefly sketch how the sample
complexity result follows from an error probability bound. For a
fuller description see \cite{borkar}, p 42.

Fix an $\epsilon>0$ such that $H^{\epsilon} := \{x: V(x) < \epsilon\} \subset \bar H^{\epsilon} := \{x: V(x) \leq \epsilon\} \subset B$. Since
$\bar B\setminus H^\epsilon$ is compact, $V(\cdot)$ is continuous,
and the o.d.e. $\dot x(t)=h(x(t))$ is well-posed, it follows that
there is a strictly positive $\Delta$ such that if the o.d.e. starts
from $x\in \bar B\setminus H^\epsilon $, flows for any time greater
than $T$ and reaches $y$, then $V(y)<V(x)-\Delta$.

Let $N_\delta(\cdot)$ denote a $\delta$-neighborhood of its
argument. Fix  $\delta$ such that $N_\delta(H^\epsilon)\subset B$,
and for all $x,y \in \bar B$ with $\|x-y\| < \delta$, we have
$\|V(x)-V(y)\| < \Delta/2$. We can do so since $V(\cdot)$ is
continuous and $\bar B$ compact.

We assume that $x_{n_0}\in B$. Further, assuming that $\rho_i <
\delta$ for all $i$, we derive an estimate $\gamma$ for the time in
which iterates, if they start with $x_{n_0}\in B\setminus
H^\epsilon$, will get trapped in $N_\delta(H^{\epsilon + \Delta/2})$
except for a small error probability given by $\sum_i\mathbb P
\left[\rho_i \geq \delta |\mathcal{B}_{i-1} \right] $.

The iterates, while they are in $B\setminus H^\epsilon$, would lose
a minimum of $\Delta$ from their potential if they could exactly
follow the o.d.e. for time $T$. As $\rho_i < \delta \ \forall i$,
over time $T$, they deviate up to $\delta$ from the o.d.e. However
since a $\delta$ shift can change the potential only by $\Delta/2$,
they are still guaranteed a loss of potential of $\Delta/2$. They
can continue losing $\Delta/2$ over every lapse of time $T$ until
$x_{n_i}\in H^\epsilon$ for some $i$.
Thereafter  the `boundary iterates' $x_{n_j}, j \geq i,$ remain
trapped in $H^{\epsilon + \Delta/2}$, since, if $x_{n_j} \in
H^{\epsilon}$ then even with the worst possible `throwing out'
$x_{n_{(j+1)}} \in H^{\epsilon+\Delta/2}$. It follows that for $j
\geq i$  the intermediate iterates $x_{n_j+m},  m < n_{(j+1)} - n_j
$, remain trapped in $N_\delta(H^{\epsilon + \Delta/2})$. Thus we
get the following estimate for $\gamma$:
\[
\gamma = \frac{\max_{x\in \bar
B}V(x)-\epsilon}{\Delta/2}\times(T+1),
\]
leading to the following sample complexity estimate
\begin{theorem}
Under assumptions (A1)-(A3), the step size assumption
(\ref{stepassume}), and the assumption that for large $u$,
 the tail probability bound (\ref{tail}) holds, we have the following bound provided $n_0$ is
sufficiently large
\[
\mathbb P\left[\bar x(t)\in N_\delta(H^{\epsilon + \Delta/2}) \
\forall t \geq t_0 + \gamma|x_{n_0}\in B\right] \geq 1 - c_1
\exp\left(-\frac{c\delta^{2/3}}{\sqrt[4]{b(n_0)}}\right),
\]
where $\delta=\delta_B/2K_T$.
\end{theorem}

\textbf{Acknowledgements:} The author would like to thank Prof.\ V.\
S.\ Borkar for much crucial help and advice while working on this
paper. The author also thanks two anonymous reviewers for pointing
out a few references, for suggesting some changes to the
presentation, and for the statement and proof of
Lemma~\ref{reviewer}.

\end{document}